\documentclass[11pt]{article}

\usepackage{amsmath}
\usepackage{amsfonts}
\usepackage{graphicx}
\usepackage{color}
\usepackage{hyperref}
\usepackage{tikz}
\usepackage{bm}

\newtheorem{theorem}{Theorem}[section]
\newtheorem{corollary}[theorem]{Corollary}

\newtheorem{lemma}[theorem]{Lemma}
\newtheorem{claim}{Claim}[section]

\newenvironment{proof}{\noindent {\bf Proof.}}{\rule{3mm}{3mm}\par\medskip}

\begin{document}
\title{{Wiener index of  unicycle graphs with given number of even degree vertices}
\author{ Peter Luo\thanks{
Morgantown High School,
Morgantown, WV 26501, email: pnluo2@stu.k12.wv.us },
~Cun-Quan Zhang\thanks{
Department of Mathematics,
West Virginia University,
Morgantown, WV 26506, email: \newline cqzhang@mail.wvu.edu},
~Xiao-Dong Zhang\thanks{
School of Mathematical Sciences,  MOE-LSC, SHL-MAC,
Shanghai Jiao Tong University,
Shanghai, P.R China, email: xiaodong@sjtu.edu.cn~~
This author was partially supported by the National Natural Science Foundation of China (Nos. 11971311 and 11531001) and  the Montenegrin-Chinese
Science and Technology Cooperation Project (No.3-12)}  }
 }        
\date{}          
\maketitle
\begin{abstract}
The Wiener index of a connected graph is the sum of the distance of all pairs of distinct vertices. It was introduced by Wiener in 1947 to analyze some aspects of branching by fitting experimental data for several properties of alkane compounds. Denote by $\mathcal{U}_{n,r}$ the set of unicyclic  graphs with $n$ vertices and $r$ vertices of even degree.  In this paper we present a structural result  on the graphs in $\mathcal{U}_{n,r}$ with minimum Wiener index and completely characterize such graphs when $ r\leq \frac{n+3}{2}$.
\end{abstract}
\section{Introduction}
Molecular descriptors can be used to describe the molecular topology of  chemical compounds.
 Topological indices of molecular graphs play a key role in the analysis of various intrinsic chemical properties, such as boiling point and specific heat.
   Wiener \cite{wiener1947}  first introduced the Wiener index when he wanted to analyze branching in organic compounds.

Let $G$ be a graph and $u \in V(G)$. Denote by $d(u,G) = \sum_{v\in V(G)} d(u,v)$ the sum of the distance  from $u$ to all other vertices  in $V(G)$. If $G$ is disconnected, denote $d(u,G) = \infty$.
 The Wiener index  $W(G)$ is defined as
   $$W(G) = \frac{1}{2} \sum_{u\in V(G)}d(u,G)  = \frac{1}{2} \sum_{u\in V(G)} \sum_{v\in V(G)}d(u,v).$$

   The Wiener index is one of the most studied topological indices in mathematical
chemistry and it is still a very active research topic (see \cite{ca2019,cavaleri2019, das2016, furtula2013-2, Georgakopoulos, gutman1993, gutman2014, gutman2016, gutman2017}). One of the fundamental problems pertaining to the Wiener index is finding extremal graphs in a graph family that achieve maximum and/or minimum values of the index in the family.  Readers are  referred to  the excellent surveys \cite{dobrynin2001, Knor2014, Knor2016,  li2006, Survey2014M}   for the developments and open problems in this area.

A graph is unicyclic (or monocyclic) if it is connected and contains
exactly one cycle. There is a rich literature on extremal problems regarding
Wiener index on unicyclic graphs (for example, see \cite{u-maximum, u2012, u-diameter, u-girth}). Section 4 of  \cite{Survey2014M}  is devoted to unicyclic graphs.

    Lin \cite{lin2014} studied the extreme Wiener index of trees with given number of even degree vertices. In this paper, we study extreme Wiener index of unicyclic graphs with a given number of even degree vertices. We give a structural result of unicyclic graphs with $n$ vertices and $r$  even degree vertices  which has the minimum Wiener index in this  family and give a complete characterization of such graphs for $r \leq \frac{n+3}{2}$.

Denote  by $\mathcal{U}_{n,r}$ the set of unicyclic  graphs with $n$ vertices and $r$ vertices of even degrees. Note that if $G \in \mathcal{U}_{n,r}$, then  $n-r$ is the number of odd degree vertices  and thus $n-r$ is even.

Let $G$ be a unicyclic graph  with the cycle $C=v_1 \dots v_\ell v_1$. Then each component $T_{v_i}$ of $G-E(C)$ containing the vertex $v_i$ is a tree rooted at $v_i$. We denote $G$ by  $H(T_1, T_2, \dots, T_\ell)$  where  $T_{v_i}$  is the tree of  $G-E(C)$ rooted at $v_i$.

A {\it subdivided star} with root $u$, denoted by $S(u)$, is a tree obtained from a star by inserting some degree $2$ vertices into each edge in the star.

Denote a {\it balanced subdivided star} with  center $v$ by $SB_v(t;b)$
 if it has $b$ branches and each branch is of length exactly $t$.  Denote an {\it almost balanced subdivided star} with center $v$ by $SaB_v(t;b)$
 if it has $b$ branches, each branch is  of length either $t$ or $t-1$,  and at least one branch is of length $t$ (See Figure \ref{star}). If the center is irrelevant or can be understood from the context, we may simply write  $SB(t;b)$ or   $SaB(t;b)$  for $SB_v(t;b)$ or   $SaB_v(t;b)$ respectively. Here, $SB(1;1)=K_2, ~ SB(0;0)=K_1$, and $SB(1;b)$ is a star.

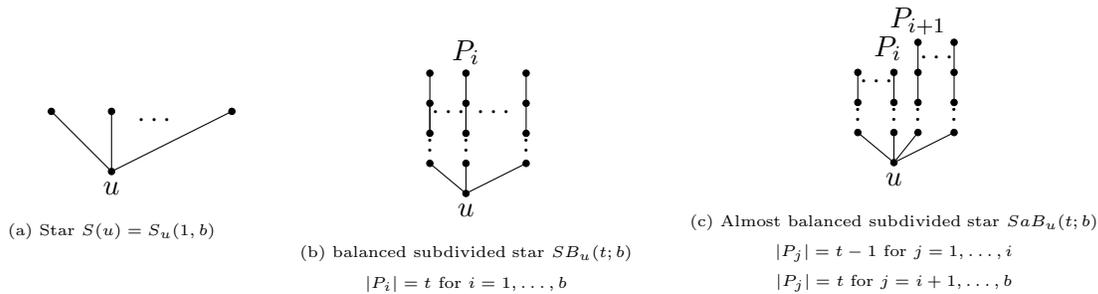
\begin{figure}
\label{Fig1}
\label{star}
\begin{center}
\begin{tikzpicture}[scale=0.8]
\draw[fill] (1,0) circle (1.5pt) node[anchor=north] {$u$};


\draw[fill] (0,1) circle (1.5pt) node[anchor=east] {};
\draw[fill] (1,1) circle (1.5pt) node[anchor=west] {};
\node at (1.75,0.85) {$\cdots$};
\draw[fill] (3,1) circle (1.5pt) node[anchor=west] {};
\draw[]  (1,0)-- (0,1) (1,0) -- (1,1)  (1,0) -- (3,1);




\node at (1,-1) {\tiny (a) Star $S(u)= S_u(1,b)$};
\node at (1,-1.5) {\mbox{}};
\node at (1,-2) {\mbox{}};
\end{tikzpicture}
~~
\begin{tikzpicture}[scale=0.8]
\draw[fill] (1,0) circle (1.5pt) node[anchor=north] {$u$};

\draw[fill] (0.4,0.5) circle (1.5pt) node[anchor=east] {};
\draw[fill] (1,0.5) circle (1.5pt) node[anchor=west] {};
\draw[fill] (2,0.5) circle (1.5pt) node[anchor=west] {};
\draw[] (0.4,0.5) -- (1,0) -- (1,0.5)  (1,0) -- (2,0.5);

\draw[fill] (0.4,1) circle (1.5pt) node[anchor=east] {};
\draw[fill] (1,1) circle (1.5pt) node[anchor=west] {};
\draw[fill] (2,1) circle (1.5pt) node[anchor=west] {};
\node at (0.4,0.85) {$\vdots$};
\node at (1,0.85) {$\vdots$};
\node at (2,0.85) {$\vdots$};

\draw[fill] (0.4,1.5) circle (1.5pt) node[anchor=east] {};
\draw[fill] (0.4,2) circle (1.5pt) node[anchor=east] {};

\draw[fill] (1,1.5) circle (1.5pt) node[anchor=west] {};
\draw[fill] (1,2) circle (1.5pt) node[anchor=west] {};
\node at (1,2.35) {$P_i$};

\draw[fill] (2,1.5) circle (1.5pt) node[anchor=west] {};
\draw[] (0.4,1.5) -- (0.4,1) (1,1.5) -- (1,1) (2,1.5) -- (2,1);

\draw[fill] (2,2) circle (1.5pt) node[anchor=west] {};
\draw[] (0.4,2) -- (0.4,1) (1,2) -- (1,1)  (2,1.5) -- (2,2);

\node at (1.5,1.35) {$\cdots$};
\node at (0.75,1.35) {$\cdots$};
\node at (1,-1) {\tiny (b) balanced subdivided star $SB_u(t;b)$};
\node at (1,-1.5) {\tiny $|P_i| = t$ for $i =1,\dots,b$};
\node at (1,-1.5) {\mbox{}};
\end{tikzpicture}
~~
\begin{tikzpicture}[scale=0.8]
\draw[fill] (1,0) circle (1.5pt) node[anchor=north] {$u$};

\draw[fill] (0.4,0.5) circle (1.5pt) node[anchor=east] {};
\draw[fill] (1,0.5) circle (1.5pt) node[anchor=west] {};
\draw[fill] (1.4,0.5) circle (1.5pt) node[anchor=west] {};
\draw[fill] (2,0.5) circle (1.5pt) node[anchor=west] {};
\draw[] (0.4,0.5) -- (1,0) -- (1,0.5) (1.4,0.5) -- (1,0) -- (2,0.5);

\draw[fill] (0.4,1) circle (1.5pt) node[anchor=east] {};
\draw[fill] (1,1) circle (1.5pt) node[anchor=west] {};
\draw[fill] (1.4,1) circle (1.5pt) node[anchor=west] {};
\draw[fill] (2,1) circle (1.5pt) node[anchor=west] {};
\node at (0.4,0.85) {$\vdots$};
\node at (1,0.85) {$\vdots$};
\node at (1.4,0.85) {$\vdots$};
\node at (2,0.85) {$\vdots$};

\draw[fill] (0.4,1.5) circle (1.5pt) node[anchor=east] {};
\draw[fill] (1,1.5) circle (1.5pt) node[anchor=west] {};
\draw[fill] (1.4,1.5) circle (1.5pt) node[anchor=west] {};
\draw[fill] (2,1.5) circle (1.5pt) node[anchor=west] {};
\draw[] (0.4,1.5) -- (0.4,1) (1,1.5) -- (1,1) (1.4,1.5) -- (1.4,1) (2,1.5) -- (2,1);

\draw[fill] (1.4,2) circle (1.5pt) node[anchor=west] {};
\draw[fill] (2,2) circle (1.5pt) node[anchor=west] {};
\draw[] (1.4,1.5) -- (1.4,2) (2,1.5) -- (2,2);

\node at (0.9, 1.9) {$P_i$};
\node at (1.4, 2.35) {$P_{i+1}$};
\node at (1.75,1.75) {$\cdots$};
\node at (0.75,1.35) {$\cdots$};
\node at (1,-1) {\tiny(c) Almost balanced subdivided star $SaB_u(t;b)$};
\node at (1,-1.5) {\tiny$|P_j| = t-1$  for $j=1,\dots,i$};
\node at (1,-2) {\tiny $|P_j| = t$  for $j=i+1,\dots,b$};
\end{tikzpicture}
\end{center}
\caption{\small\it  Subdivided stars}
\end{figure}

We first present a structure of graphs in ${\cal U}_{n,r}$ with minimum Wiener index.

\begin{theorem}
\label{TH: Conf}
 Let $   r\geq 0$ and $ G\in {\cal U}_{n,r}$ such that $W(G)$ is minimum among all graphs in ${\cal U}_{n,r}$. We have the following:

   (1)   if $r\leq 2$ and $G$ must be  $H(SB(1; b_1), X_1, X_2)$ where $X_j \in \{ K_1, K_2 \}$ for each $j = 1,2$, $b_1$ is odd. See Figure~2-(b-d).

    (2) if $r\geq 3$, then  $G$   must be $H(K_2,K_1,K_1,K_1,K_1)$ or   $H(SaB(t;b_2),  K_1, \dots, K_1)  $ where  $b_2 \geq 0$ is even and $t \geq 0$. See Figure~2(a) and (e).
 \end{theorem}

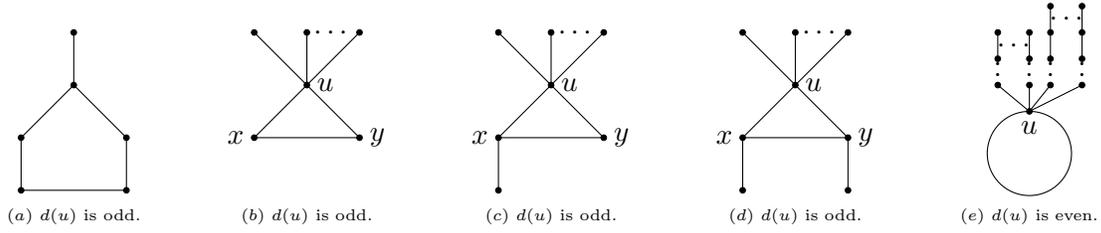
\begin{figure}
\label{Fig2}
\begin{center}
\begin{tikzpicture}[scale=0.7]
\draw[fill] (0,0) circle (1.5pt) node[anchor=east] {};
\draw[fill] (1,1) circle (1.5pt) node[anchor=west] {};
\draw[fill] (2,0) circle (1.5pt) node[anchor=west] {};
\draw[fill] (0,-1) circle (1.5pt) node[anchor=east] {};
\draw[fill] (2,-1) circle (1.5pt) node[anchor=west] {};
\draw[fill] (1,2) circle (1.5pt) node[anchor=west] {};
\draw[] (0,0) -- (1,1) -- (2,0) -- (2,-1)--(0,-1)--(0,0) (1,1)--(1,2);
\node at (1,-1.5) {\tiny $(a)~d(u)$ is odd.};
\end{tikzpicture}
\qquad
\begin{tikzpicture}[scale=0.7]
\draw[fill] (0,0) circle (1.5pt) node[anchor=east] {$x$};
\draw[fill] (1,1) circle (1.5pt) node[anchor=west] {$u$};
\draw[fill] (2,0) circle (1.5pt) node[anchor=west] {$y$};
\draw[] (0,0) -- (1,1) -- (2,0) -- (0,0);
\draw[fill] (0,2) circle (1.5pt) node[anchor=east] {};
\draw[fill] (1,2) circle (1.5pt) node[anchor=west] {};
\draw[fill] (2,2) circle (1.5pt) node[anchor=west] {};
\draw[] (0,2) -- (1,1) -- (1,2) (1,1) -- (2,2);
\node at (1.5,2) {$\cdots$};
\node at (1,-1.5) {\tiny $(b)~d(u)$ is odd.};
\end{tikzpicture}
\qquad
\begin{tikzpicture}[scale=0.7]
\draw[fill] (0,0) circle (1.5pt) node[anchor=east] {$x$};
\draw[fill] (1,1) circle (1.5pt) node[anchor=west] {$u$};
\draw[fill] (2,0) circle (1.5pt) node[anchor=west] {$y$};
\draw[fill] (0,-1) circle (1.5pt) node[anchor=west] {};
\draw[] (0,0) -- (1,1) -- (2,0) -- (0,0) (0,0) -- (0,-1);
\draw[fill] (0,2) circle (1.5pt) node[anchor=east] {};
\draw[fill] (1,2) circle (1.5pt) node[anchor=west] {};
\draw[fill] (2,2) circle (1.5pt) node[anchor=west] {};
\draw[] (0,2) -- (1,1) -- (1,2) (1,1) -- (2,2);
\node at (1.5,2) {$\cdots$};
\node at (1,-1.5) {\tiny $(c) ~d(u)$ is odd.};
\end{tikzpicture}
\qquad
\begin{tikzpicture}[scale=0.7]
\draw[fill] (0,0) circle (1.5pt) node[anchor=east] {$x$};
\draw[fill] (1,1) circle (1.5pt) node[anchor=west] {$u$};
\draw[fill] (2,0) circle (1.5pt) node[anchor=west] {$y$};
\draw[fill] (0,-1) circle (1.5pt) node[anchor=west] {};
\draw[fill] (2,-1) circle (1.5pt) node[anchor=west] {};
\draw[] (0,0) -- (1,1) -- (2,0) -- (0,0) (0,0) -- (0,-1) (2,0) -- (2,-1);
\draw[fill] (0,2) circle (1.5pt) node[anchor=east] {};
\draw[fill] (1,2) circle (1.5pt) node[anchor=west] {};
\draw[fill] (2,2) circle (1.5pt) node[anchor=west] {};
\draw[] (0,2) -- (1,1) -- (1,2) (1,1) -- (2,2);
\node at (1.5,2) {$\cdots$};
\node at (1,-1.5) {\tiny $(d) ~d(u)$ is odd.};
\end{tikzpicture}
\qquad
\begin{tikzpicture}[scale=0.7]
\draw[] (1,-0.3) circle [radius=0.8];
\draw[fill] (1,0.5) circle (1.5pt) node[anchor=north] {$u$};

\draw[fill] (0.4,1) circle (1.5pt) node[anchor=east] {};
\draw[fill] (1,1) circle (1.5pt) node[anchor=west] {};
\draw[fill] (1.4,1) circle (1.5pt) node[anchor=west] {};
\draw[fill] (2,1) circle (1.5pt) node[anchor=west] {};
\draw[] (0.4,1) -- (1,0.5) -- (1,1) (1.4,1) -- (1,0.5) -- (2,1);

\draw[fill] (0.4,1.5) circle (1.5pt) node[anchor=east] {};
\draw[fill] (1,1.5) circle (1.5pt) node[anchor=west] {};
\draw[fill] (1.4,1.5) circle (1.5pt) node[anchor=west] {};
\draw[fill] (2,1.5) circle (1.5pt) node[anchor=west] {};
\node at (0.4,1.35) {$\vdots$};
\node at (1,1.35) {$\vdots$};
\node at (1.4,1.35) {$\vdots$};
\node at (2,1.35) {$\vdots$};

\draw[fill] (0.4,2) circle (1.5pt) node[anchor=east] {};
\draw[fill] (1,2) circle (1.5pt) node[anchor=west] {};
\draw[fill] (1.4,2) circle (1.5pt) node[anchor=west] {};
\draw[fill] (2,2) circle (1.5pt) node[anchor=west] {};
\draw[] (0.4,2) -- (0.4,1.5) (1,2) -- (1,1.5) (1.4,2) -- (1.4,1.5) (2,2) -- (2,1.5);

\draw[fill] (1.4,2.5) circle (1.5pt) node[anchor=west] {};
\draw[fill] (2,2.5) circle (1.5pt) node[anchor=west] {};
\draw[] (1.4,2) -- (1.4,2.5) (2,2) -- (2,2.5);

\node at (1.75,2.25) {$\cdots$};
\node at (0.75,1.75) {$\cdots$};
\node at (1,-1.5) {\tiny $(e)~d(u)$ is even.};
\end{tikzpicture}
\end{center}
\caption{\small\it  Configurations for Theorem~\ref{TH: Conf}}
\end{figure}

As a corollary, we  can further completely characterize $G \in \mathcal{U}_{n,r}$ such that $W(G)$ is the minimum when $0\leq r \leq \frac{n+3}{2}$.

\begin{theorem}
\label{TH: small r}
Let $0\leq r \leq \frac{n+3}{2}$.
Then $W(G)$ is minimum among all graphs in ${\cal U}_{n,r}$ if and only if $G$ is one of the following:

(a)  $H(K_2,K_1,K_1,K_1,K_1)$, or  $H(SB(1; b_1), X_1, X_2)$ where $b_1\geq 1$ is odd  and $X_j \in \{ K_1, K_2 \}$ for each $j$. See Figure ~2-(a-d).

(b) $H(SB(1; b_2), K_1, K_1), ~H(SB(1;b_3), K_1, K_1, K_1 ), ~H(SaB(2;b_4),K_1, K_1, K_1, K_1)$,
 where  $b_i$ is even for each $i=2,3,4$. See Figure ~3-(a-c).
\end{theorem}
\begin{figure}
\label{Fig3}
\begin{center}
\begin{tikzpicture}[scale=0.8]
\draw[fill] (0,0) circle (1.5pt) node[anchor=east] {$x$};
\draw[fill] (1,1) circle (1.5pt) node[anchor=west] {$u$};
\draw[fill] (2,0) circle (1.5pt) node[anchor=west] {$y$};
\draw[] (0,0) -- (1,1) -- (2,0) -- (0,0);
\draw[fill] (0,2) circle (1.5pt) node[anchor=east] {};
\draw[fill] (1,2) circle (1.5pt) node[anchor=west] {};


\draw[fill] (0.5,2) circle (1.5pt) node[anchor=west] {};
\draw[fill] (2,2) circle (1.5pt) node[anchor=west] {};
\draw[] (0,2) -- (1,1) -- (1,2) (1,1) -- (2,2) (0.5,2)--(1,1);
\node at (1.5,2) {$\cdots$};
\node at (1,-1.5) {$(a)$};
\end{tikzpicture}
\qquad
\begin{tikzpicture}[scale=0.8]
\draw[fill] (0,0) circle (1.5pt) node[anchor=east] {$x$};
\draw[fill] (1,1) circle (1.5pt) node[anchor=west] {$u$};
\draw[fill] (2,0) circle (1.5pt) node[anchor=west] {$y$};
\draw[fill] (1,-1) circle (1.5pt) node[anchor=west] {$v$};
\draw[] (0,0) -- (1,1) -- (2,0) -- (1,-1)--(0,0);
\draw[fill] (0,2) circle (1.5pt) node[anchor=east] {};
\draw[fill] (1,2) circle (1.5pt) node[anchor=west] {};

\draw[fill] (0.5,2) circle (1.5pt) node[anchor=west] {};
\draw[fill] (2,2) circle (1.5pt) node[anchor=west] {};
\draw[] (0,2) -- (1,1) -- (1,2) (1,1) -- (2,2) (0.5,2)--(1,1);
\node at (1.5,2) {$\cdots$};
\node at (1,-1.5) {$(b)$};
\end{tikzpicture}
\qquad
\begin{tikzpicture}[scale=0.8]
\draw[fill] (0,0) circle (1.5pt) node[anchor=east] {$x$};
\draw[fill] (1,1) circle (1.5pt) node[anchor=west] {$u$};
\draw[fill] (2,0) circle (1.5pt) node[anchor=west] {$y$};
\draw[fill] (0,-1) circle (1.5pt) node[anchor=east] {$v$};
\draw[fill] (2,-1) circle (1.5pt) node[anchor=west] {$w$};
\draw[] (0,0) -- (1,1) -- (2,0) -- (2,-1)--(0,-1)--(0,0);

\draw[fill] (-0.5,1.5) circle (1.5pt) node[anchor=west] {};
\draw[fill] (0.5,1.5) circle (1.5pt) node[anchor=east] {};
\draw[fill] (1,1.5) circle (1.5pt) node[anchor=west] {};
\draw[fill] (2,1.5) circle (1.5pt) node[anchor=west] {};
\draw[] (-0.5,1.5) -- (1,1) -- (0.5,1.5) (1,1) -- (2,1.5) (0.5,1.5)--(1,1)--(1,1.5);
\node at (0,1.5) {$\cdots$};

\draw[fill] (1,2) circle (1.5pt) node[anchor=west] {};
\draw[fill] (2,2) circle (1.5pt) node[anchor=west] {};
\draw[] (1,1.5) -- (1,2) (2,1.5) -- (2,2);
\node at (1.5,1.75) {$\cdots$};
\node at (1,-1.5) {$(c)$};
\end{tikzpicture}
\end{center}
\caption{\small\it  Configurations for Theorem~\ref{TH: small r}: $d(u) = n-r +2$ even.}
\end{figure}
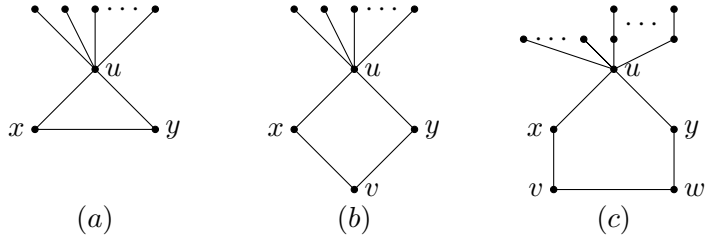

\section{Lemmas}
In this section, we introduce some lemmas that will be needed in the proof of our theorems.

Let $G$ be a graph  which has two distinct cut vertices $u,v$.  Then $G$ consists of three components $X, Y$ and $H$ such that $X$ and $H$ share $u$ and $Y$ and $H$ share $v$.
Denote by $G_{X\to Y}$ to be the graph obtained from $G$ by moving $X$ to $Y$ so that $X,Y$ and $H$ share the vertex $v$. $G_{Y\to X}$ is defined similarly (see Figure~4).

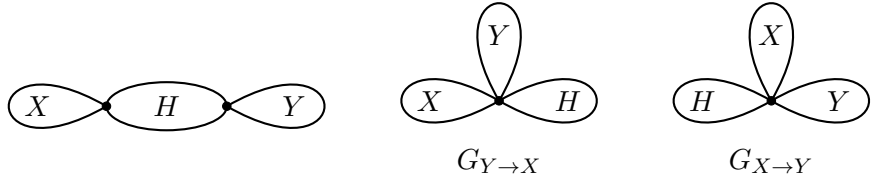
\begin{figure}
\label{Fig4}
\begin{center}
\begin{tikzpicture}[scale=0.8]
\draw[thick, min distance=25mm, in=210, out=150, loop, looseness=10] (-1,0)  to (-1,0)  node [left=6mm] {$X$};
\draw[thick] (0,0) ellipse (10mm and 4mm) node [] {$H$};
\draw[thick, min distance=25mm, in=330, out=30, loop, looseness=10] (1,0)  to (1,0) node [right=6mm] {$Y$};
\draw[fill] (-1,0) circle (2pt) node[anchor=west] {};
\draw[fill] (1,0) circle (2pt) node[anchor=west] {};
\end{tikzpicture}
\begin{tikzpicture}[scale=0.8]
\draw[thick, min distance=25mm, in=210, out=150, loop, looseness=10] (0,0)  to (0,0)  node [left=6mm] {$X$};
\draw[thick, min distance=25mm, in=120, out=60, loop, looseness=10] (0,0)  to (0,0)  node [above=6mm] {$Y$};
\draw[thick, min distance=25mm, in=330, out=30, loop, looseness=10] (0,0)  to (0,0) node [right=6mm] {$H$};
\draw[fill] (0,0) circle (2pt) node[anchor=west] {};
\node at (0,-1) {${G_{Y\to X}}$};
\end{tikzpicture}
\begin{tikzpicture}[scale=0.8]
\draw[thick, min distance=25mm, in=210, out=150, loop, looseness=10] (0,0)  to (0,0)  node [left=6mm] {$H$};
\draw[thick, min distance=25mm, in=120, out=60, loop, looseness=10] (0,0)  to (0,0)  node [above=6mm] {$X$};
\draw[thick, min distance=25mm, in=330, out=30, loop, looseness=10] (0,0)  to (0,0) node [right=6mm] {$Y$};
\draw[fill] (0,0) circle (2pt) node[anchor=west] {};
\node at (0,-1) {${G_{X\to Y}}$};
\end{tikzpicture}
\end{center}
\caption{\small\it  $G_{X\to Y}$ and $G_{Y\to X}$.}
\end{figure}

The following two lemmas are  applied frequently in the  our proofs.

\begin{lemma}(\cite{HLW2011})
\label{LE:two cut vertices}
Let $G$ be a graph  which has two distinct cut vertices $u,v$. Let $X, Y$ and $H$ be the three components such that $X$ and $H$ share $u$ and $Y$ and $H$ share $v$. Then
either $W(G_{X\to Y}) < W(G)$ or $W(G_{Y\to X}) < W(G)$.
\end{lemma}

An edge  $e=uv$ is a bridge if its removal increases the number of components  and a bridge is trivial  if one of its ends is a leaf vertex.

 In particular when $H$ is a single edge, $G_{X\to Y} = G_{Y\to X}$. Thus we have the following corollary.

\begin{corollary}
\label{COR:bridge}
Let $e=uv$ be  an nontrivial bridge of $G$ and let $X$ and $Y$ be two components of $G-e$.  Let $H= e$.   Then $G_{X\to Y} = G_{Y\to X}$ and thus  $W(G_{X\to Y}) < W(G)$. In particular, if  one of $d(u)$ and $d(v)$ is odd, $G$ and $G_{X\to Y}$ have the same number of even degree vertices.
\end{corollary}

\begin{lemma}(\cite{PB1986})
\label{LE:one cut vertex}
Let $G_1$ and $G_2$ be two disjoint graphs and $G$ be the graph obtained  by identifying $G_1$ and $G_2$ at a vertex $u$. Then
\[
W(G) = W(G_1) + W(G_2) + (|V(G_1)| -1)d(u, G_2) + (|V(G_2)| -1)d(u, G_1),
\]
where $d(u,G_i) = \sum_{v\in V(G_i)} d_{G_i}(u,v)$.
\end{lemma}

The next two lemmas are well-known.

\begin{lemma}
\label{Le:cycle}
Let $C$ be a cycle of length $g$. Then

$$W(G) = \left\{
\begin{array}{ll}
k^3  \hskip 2cm~~~\mbox{if $g = 2k$},\\
\frac{k(k+1)(2k+1)}{2}~~~~~~\mbox{if $g = 2k+1$}.
\end{array}
\right.
$$
\end{lemma}

\begin{lemma}
Let $P_t$ be a path with $t$ edges. Then
$$W(P_t) =  \frac{t(t+1)(t+2)}{6}.$$
\end{lemma}

Let $e=uv$ be an edge in $G$. Contracting the edge $e$  means deleting the edge $uv$ and then identifying  $u$ and $v$.

The following lemma is a key lemma in the proof of our main results.

\begin{lemma}
\label{Le:shorter cycle}
Let $G$ be a unicyclic graph with the cycle $C$  of  length at least $4$  and maximum degree at least $3$ satisfying:  $G \not = H(K_2, K_1, K_1,  K_i, K_1 )$ for each  $i=1,2$,  every vertex $x$ in $C$ has either degree 2 or  odd degree adjacent to exactly $d_G(x)-2$ leaf vertices, and at most one vertex in $C$  has degree at least $4$.
Let $uv$ be an edge in $C$ such that

\noindent
 (i)  $d_G(u)$ is the maximum;

\noindent
  (ii)  $d_G(u)+d_G(v)$ is also maximum subject to (i).

  Let $G_1$ be the graph obtained from $G$ by contracting the edge $uv$ and adding a leaf adjacent to $u$.  Then

(1) $|V(G)| = |V(G_1)|$ and $G_1$ and $G$ have  the same number of even degree vertices;

(2) $W(G_1) < W(G)$.
\end{lemma}

\begin{proof}
Denote $C= u_1u_2\dots u_gu_1$ and $k = \lfloor\frac{g}{2}\rfloor$. Without loss of generality, assume  that $u_2u_1 =uv$ is the edge satisfying (i) and (ii) where $u=u_2$ and $v=u_1$.    Let  $u_i'$ be a leaf  vertex adjacent to $u_i$  in $G$ if $d_G(u_i) \geq 3$.   Note that every vertex $x\not = u_2$ in $C$ is adjacent to at most one leaf  vertex in $G$ and   in $G_1$, $u_1$ is a leaf  vertex adjacent to $u_2$ and if $d_G(u_1) = 3$, then $u_1'$ is also a leaf vertex adjacent to $u_2$.

Obviously $G$ and $G_1$ have the same number of vertices.

Note that $d_G(u_2)$ is odd since $d_G(u_2) \geq 3$ is odd by the assumption. If $d_G(u_1) = 2$, then $d_{G_1}(u_2) = d_G(u_2) + 1$ is even and $d_{G_1}(u_1) = 1$.  If $d_G(u_1) = 3$, then $d_{G_1}(u_2) = d_G(u_2)  +2$ remains odd and $d_{G_1}(u_2)  = d_{G_1}(u_2')= 1$. Thus $G_1$ and $G$ have the same number of even degree vertices.  This proves (1).

 \medskip \noindent
Now we are going to prove (2): $W(G_1) < W(G)$.

We first have the following observations.
\begin{claim}
\label{claim1}
(1) $d_{G_1}(x,y) \leq d_G(x,y)$ for any two distinct vertices  with  $u_1\not  \in \{x,y\}$.

(2) If $x\not = u_1$, $d_{G_1}(x,G_1\setminus\{u_1\}) \le d_{G}(x,G\setminus\{u_1\})$.

(3) $\displaystyle \sum_{2\leq i < j \leq g} (d_G(u_i,u_j)-d_{G_1}(u_i,u_j)) =  \frac{k(k-1)}{2}$.
\end{claim}
\begin{proof}
(1) and (2) are obvious by the definition of $G_1$.     Now we prove (3).

Since $d_G(u_1,C_g) = \frac{2W(C_g)}{g}$, we have
$$\sum_{2\leq i < j \leq g} (d_G(u_i,u_j)-d_{G_1}(u_i,u_j)) = W(C_g)- d_G(u_1,C_g) -W(C_{g-1}) = \frac{(g-2)}{g}W(C_g)-W(C_{g-1}).$$

If $g = 2k +1$, then   by Lemma~\ref{Le:cycle}, we have
 $$\frac{g-2}{g}W(C_g)-W(C_{g-1}) = \frac{(2k-1)}{2k+1}\frac{k(k+1)(2k+1)}{2}-k^3 = \frac{k(k-1)}{2}.
 $$
 If  $g = 2k$, then $$\frac{g-2}{g}W(C_g)-W(C_{g-1}) = \frac{(2k-2)}{2k}k^3-\frac{(k-1)k(2k-1)}{2} = \frac{k(k-1)}{2}.
 $$
\end{proof}

 Denote $A = \cup_{i=3}^{k}N[u_i]\setminus\{u_2, u_{k+2}\}$,  $B = \cup_{i=k+3}^{g}N[u_i]\setminus\{u_1, u_{k+2}\}$,  $C=N_G[u_{k+2}]\setminus \{u_{k+1}, u_{k+3}\}$, and $D=N_G[u_2]\setminus \{u_1,u_3\}$. That is, $A$ consists of $u_3,\dots,u_{k+1}$ together with their respective leaf neighbors if they exist,  $B$ consists of $u_{k+3},\dots, u_g,$ with their similarly defined leaf neighbors,     $C$  consists of  $u_{k+2}$ together with its leaf neighbor (if it exists), and  $D$  consists of  $u_{2}$ together with its leaf neighbor (if it exists).

 \begin{claim}
 \label{claim2}
 We have the following statements.

 (1) For each $v\in A$ and each leaf neighbor $w$ of $u_2$, $d_{G_1}(u_1,v) = d_{G}(u_1,v)$ and $d_{G_1}(w,v)= d_{G}(w,v)$.

 (2) For each $v\in B$ and each leaf neighbor $w \in D$, $d_{G_1}(u_1,v) = d_{G}(u_1,v)+1$ and $d_{G_1}(w,v)= d_{G}(w,v)-1$.

 (3) For each $v\in C$ and each leaf neighbor $w$ of $u_2$, we have

  (3-1)  $d_{G_1}(u_1,v) = d_{G}(u_1,v)+1$;

  (3-2)  $d_{G_1}(w,v) = d_{G}(w,v)$ if $g = 2k+1$ and $d_{G_1}(w,v) = d_{G}(w,v) -1$ if $g = 2k$.
 \end{claim}
 \begin{proof} By the definition of $G_1$, we have
 \[
  d_{G_1}(u_2,v) =  \left\{
\begin{array}{ll}
d_{G}(u_2,v)   ~~~~~~~~~~~\mbox{if $v\in A$ or if $v\in C$ when $g = 2k+1$},\\
d_{G}(u_2,v) -1 ~~~~~~\mbox{if $v\in B$ or if $v\in C$ when $g=2k$}.
\end{array}
\right.
\]

By the definition of $G_1$ again, we have   that if $w=u_1$ or if $w$ is a leaf neighbor of $u_2$, then  for every $v\in A\cup B\cup C$,
$$d_{G_1}(w,v) =   1 + d_{G_1}(u_2,v).$$
This proves the claim.
 \end{proof}

 \begin{claim}
 \label{claim3}
 If  $d_G(u_1) = 3$, then  $\sum_{i=1}^{k+1}(d_G(u_1',u_i) -  d_{G_1}(u_1',u_i))  = -1 + k > 0$.
  \end{claim}
\begin{proof}
 It follows from the following:

  \medskip
   $\bullet$ $d_G(u_1',u_1) =1 =  d_{G_1}(u_1',u_1) - 1$.

    \medskip
   $\bullet$  $d_G(u_1',u_i) =  d_{G_1}(u_1',u_i)  + 1$ for each $i = 2,\dots, k+1$.
   \end{proof}

  By Claims~\ref{claim1}-\ref{claim3},  if $g = 2k$ , we have
 $W(G) - W(G_1) \geq \frac{k(k-1)}{2} > 0$;

  And if $g = 2k+1$, then  $W(G) - W(G_1) \geq  \frac{k(k-1)}{2} - 1 \geq 0$.    Thus if $k\geq 3$, then $W(G) -W(G_1) > 0$.

 Now assume $k=2$ and $g = 5$.  By Claim~\ref{claim2}, we have the following:

 $\bullet$ $\sum_{x\in B\cup C}(d_G(u_1,x)-d_{G_1}(u_1,x)) = \sum_{x\in B\cup C} (-1) = - |B\cup C| = -|B| - |C|$

 $\bullet$ $\sum_{x\in B, y\in D}(d_G(x,y)-d_{G_1}(x,y)) = \sum_{x\in B, y\in D} (1) = |B||D|$

 $\bullet$  If $d_G(u_1) = 3$, then  $d_{G}(u_1,u_1') = 1 + d_{G_1}(u_1,u_1') $ and $\sum_{x\in D\cup A} (d_G(x,u_1')-d_{G_1}(x,u_1'))= |A| + |D|$.

Note that  by the hypothesis, $|D| \geq |B| $ and $|D| \geq |B|$. Thus  if $d_G(u_1) = 3$,
$$W(G)-W(G_1)\ge |B||D|+|D| + |A|+1-|B|-|C|>0.$$

If $d_G(u_1) = 2$, then $d_G(u_3) = 2$ and $d_G(u_2) + d_G(u_1) \geq d_G(u_4) + d_G(u_5)$.  Hence $|D| +2 \geq |B| + |C|$. Moreover since $G \not  = H(K_2, K_1,  K_1, K_i, K_1 )$ for  each $i=1,2$, $d_G(u_2) \geq 5$. This implies $2\leq |B| \leq 3$ and $2\leq |C|\leq 3$. Therefore
 $$W(G)-W(G_1)\ge |B||D|-|B|-|C|>0.$$
 This completes the proof of the lemma.
     \end{proof}

\section{Proof of Theorem~\ref{TH: Conf}}
We will complete the proof of Theorem~\ref{TH: Conf} in this section.

We first need to introduce more notations.
 Let $G$ be a unicyclic graph with the cycle  $C$.   Recall that  $G-E(C)$ is a forrest such that each component is a tree rooted at a vertex $u \in V(C)$. Denote such a tree by $T_u$ and denote  by $G-T_u$ the graph obtained from $G$ by deleting all edges in $E(T_u)$ and all resulting isolated vertices. That is $G-T_u$  is the graph obtained from $G$ by deleting all vertices in $T_u$ except $u$. Similarly we can define $G-T_u-T_v$. $T_u$ is called trivial if $T_v = K_1$, i.e. $d_G(v) = 2$.

 Let $T$ be a  tree  rooted at $u$. Let  $v$ be a vertex in $T$ and $S$ be a set of  some children of $v$. Denote  by $T_v(S)$ the subtree rooted at $v$  induced by $\{v\}\cup S$ and the descendents of $S$.  If $S = \{x\}$ or $S= \{x,y\}$, we simply denote it by    $T_v(x)$ or $T_v(x,y)$, respectively.

Now we are ready to prove Theorem~\ref{TH: Conf}.

\medskip
\noindent
{\bf Proof of Theorem~\ref{TH: Conf}.} Let $G \in  \mathcal{U}_{n,r}$ such that $W(G)$ is as small as possible.  Let $C = u_1u_2\dots u_gu_1$ be the cycle in $G$.

\begin{claim}
\label{odd cut}
Let $e =xy $ be a nontrivial bridge, meaning $\min\{d_G(x), d_G(y)\}\geq 2$. Then both $d(x)$ and $d(y)$ are even.
\end{claim}
\begin{proof}
Let $H = e$, $X$ and $Y$ be two components of $G-e$. If one of $d(x)$ and $d(y)$ is odd, by Corollary~\ref{COR:bridge}, $G_{X\to Y} \in \mathcal{U}_{n,r}$ and $W(G_{X\to Y}) < W(G)$, a contradiction to the minimality of $W(G)$. Thus both $d(x)$ and $d(y)$ are even.
\end{proof}

\begin{claim}
\label{odd star}
Let $u \in V(C)$.

(1) If $d_G(u)$  is odd,  then $T_u$ is a star.

(2) If $d_G(u)\geq 4$ is even, then $T_u$ is a subdivided star.
\end{claim}
\begin{proof}
(1) Assume that $d(u)$ is odd. Then $d(u) \geq 3$. Suppose to the contrary that $T_u$ is not a star. Then $d(u) \geq 3$ and there is an edge $uv \in T_u$ such that  $uv$ is a nontrivial bridge in $G$, a contradiction to Claim~\ref{odd cut}.

(2) Assume  that $d(u) \geq 4$ is even.  Suppose to the contrary that $d(v) \geq 3$  for some vertex $v \in V(T_u) \setminus \{u\}$. Let $X$ be a subtree rooted at $v$  such that $d_X(v) = 2$. Let $Y= G-T_u$ and $H = T_u- (X-v)$.  Since $d_X(v) = 2$ and $d_Y(u) = 2$ both are even, both $G_{X\to Y}$ and $G_{Y\to X}$ belong to   $\mathcal{U}_{n,r}$. And by Lemma~\ref{LE:two cut vertices},
either $W(G_{X\to Y}) < W(G)$ or  $W(G_{Y\to X})< W(G)$, a contradiction to the minimality of $W(G)$.
\end{proof}

\begin{claim}
\label{one big vertex}
There are at most  one vertex in $V(C)$ with  at  least $4$.
\end{claim}
\begin{proof}
Suppose to the contrary that $u$ and $v$ are two vertices in $V(C)$ with degree at least $4$. Let $X$ be a subtree of $T_u$ rooted at $u$ such that $d_X(u) = 2$ and $Y$ be a subtree of $T_v$ rooted at $v$ such that $d_Y(v) = 2$.   Let $H= G- (X-u) - (Y-v)$. Then both $G_{X\to Y}$ and $G_{Y\to X}$ belong to $\mathcal{U}_{n,r}$. By Lemma~\ref{LE:two cut vertices},  either $W(G_{X\to Y}) < W(G)$ or $W(G_{Y\to X}) < W(G)$, a contradiction to the minimality of $W(G)$.
\end{proof}

  Without loss of generality, assume $d_G(u_1) =  \max \{d_G(u)| u \in V(C)\}$. We consider three cases according to $d_G(u_1)$ in the following.

\subsection{ The case when $d_G(u_1) = \max \{d_G(u)| u \in V(C)\} = 2$}

In this case, $g = |V(G)|$ and  $G = C_n$ with $r  = n$. $C_n$ is the only graph in $\mathcal{U}_{n,n}$.  Thus the theorem is true.

\subsection{The case when $d_G(u_1) = \max \{d_G(u)| u \in V(C)\}  \geq 3$ is odd }

\begin{claim}
\label{triangle}
 Either $g=5$ and $G = H(K_2,K_1,K_1,K_1,K_1)$ or $g = 3$ and  thus $r \leq 2$.  In particular, $d_G(u_2) = d_G(u_3) = 2$ if $r = 2$ or $d_G(u_2) = 2$ and $d_G(u_3) = 3$ if $r = 1$, or $d_G(u_2) = d_G(u_3) = 3$ if $r = 0$.
\end{claim}
\begin{proof} If $G = H(K_2,K_1,K_1,K_2,K_1)$, it is easy to see that $W(G)<W(G_1)$ by simple calculation, where $G_1=H(K_{1,2}, K_2, K_1, K_1)$.
Suppose to the contrary $g \geq 4$ and $G \not = H(K_2,K_1,K_1,K_i,K_1)$  for $i=1,2.$  Let $G_1$ be the graph obtained from $G$ by contracting $u_1u_2$,   where  $d_G(u_2)$ is the maximum and   $d_G(u_2)+d_G(u_1)\ge d(x)+d(y)$ for any edge $xy$ with $d_G(x)=d_G(u_2)$.
                   By Lemma~\ref{Le:shorter cycle}, $W(G_1)< W(G)$ and $G_1 \in \mathcal{U}_{n,r}$, a contradiction to the minimality of $W(G)$. Thus $g = 3$.  By Claim~\ref{odd star}, $T_{u_1}$ is a star. Since $d_G(u_1) \geq 3$ is odd, we have $r \leq 2$.
\end{proof}

\medskip
Claim~\ref{triangle} completes the proof of (1) in the theorem.

\subsection{The case when $d(u_1) = \max \{d(u)| u \in V(C)\}  \geq 4$ is even}

\begin{claim}
\label{one big even vertex}
 $d(u_i) = 2$ for each $ i = 2,\dots, g$.
\end{claim}
\begin{proof}
Suppose to the contrary that  $d(u_i) \geq 3$ for some $i \in 2,\dots, g$.  By Claim~\ref{one big vertex}, $d(u_i) = 3$.  Let $u_i'$ be the leaf neighbor of $u_i$. Since $d(u_1) \geq 4$, let $X = T_{u_1}$,  $Y = T_{u_i}$  and $H= G- (X-u_1) - (Y-u_i)$. Then both $G_{X\to Y}$ and $G_{Y\to X}$ belong to $\mathcal{U}_{n,r}$. By Lemma~\ref{LE:two cut vertices},  either $W(G_{X\to Y}) < W(G)$ or $W(G_{Y\to X}) < W(G)$, a contradiction to the minimality of $W(G)$.
\end{proof}

To prove that $G$ must be the configuration (2)  in Theorem~\ref{TH: Conf}, we need to show  the following claim which says that the subdivided star $T_{u_1}$ is almost balanced.

\begin{claim}
\label{almost balanced}
Let $P_1$ and $P_2$ be the two branches of $T_{u_1}$.  Then $||P_1| - |P_2||\leq 1$.
\end{claim}
\begin{proof}
Suppose to the contrary $|P_1| - |P_2| \geq 2$. Let  $v_1$ and $v_2$ be the other endvertices of $P_1$ and $P_2$  than $u_1$, respectively.

Denote $P_1' = P_1-v_1$ and $P_2' = P_2 + v_2v_1$.  Let $G_1$ be the graph obtained from $G$ by replacing $P_1$ and $P_2$ with $P_1'$ and $P_2'$.  Then $G_1 \in  \mathcal{U}_{n,r}$.

We have the following facts:

\medskip \noindent
 (a)   $|P_1| = d_G(u_1,v_1) \geq |P_2| + 2 >  d_{G_1}(u_1,v_1) = |P_2| + 1$.

 \medskip \noindent
(b)   for any two vertices $x,y \in V(G) = V(G_1)$, if $v_1\not \in \{x,y\}$, then $d_G(x,y) = d_{G_1}(x,y)$.

\medskip \noindent
 (c) For any $x \in V(G)- V(P_1)-V(P_2)$, $d_G(x,v_1) = d_G(x,u_1) + d_G(u_1v_1) > d_{G_1}(x,u_1)  + d_{G_1}(u_1,v_1)$.

 \medskip \noindent
(d) $\sum_{x\in P_1'\cup P_2'} d_{G_1}(x,v_1) = \sum_{x\in P_1\cup P_2} d_{G}(x,v_1)$ since   $P_1'\cup P_2'$  and $P_1\cup P_2$  are two paths with the same length and have $v_1$ as one endvertex.

 \medskip \noindent
By (a)$-$(d)  we have $W(G_1) < W(G)$, a contradiction to the minimality of $W(G)$. This completes the proof of the theorem.
\end{proof}

\section{Proof of Theorem~\ref{TH: small r}}
 We will complete the proof of Theorem~\ref{TH: small r} in this section
We first prove the following two lemmas.

\begin{lemma}
\label{t=2}
Let $G =H(SaB_u(t,b), K_1,\dots,K_1) \in  \mathcal{U}_{n,r}$   with girth $g$,  $3\le r \leq \frac{n+3}{2}$ and $d_G(u) = n-r + 2$ is even.
Then

(a) $t\leq 2$.

(b)  $u$ is adjacent to a leaf  vertex unless $r = \frac{n+3}{2}$ and $g= 3$ in which case, $ SaB_u(2;b)=SB_u(2,b)$ is  balanced with $t=2$.
\end{lemma}
\begin{proof}
We rearrange the inequality  $r\leq \frac{n+3}{2}$ to get $n-2(n-r)\leq 3$.  Thus   $2(n-r) + 3 \geq n \geq (t-1)(n-r) + 1 + g$.  Since $g\geq 3$, we have  $(3-t)(n-r) \geq 1 > 0$. Therefore $t\leq 2$.

 Note that $3$ is the minimum value for the length of a cycle and that $2n-2r$ represents the number of vertices in $SaB_u(2;b).$ Thus, equality   $n-2(n-r)=3$ is achieved when $SaB_u(2,b)=SB_u(2;b)$ and $g=3.$
\end{proof}

We first define an operation as follows:

 Let $G =H(SaB_u(t;b), K_1,\dots,K_1) \in  \mathcal{U}_{n,r}$ where $3\le r \leq \frac{n+3}{2}$  and $g\geq 4$.   By Lemma~\ref{t=2}, $t\leq 2$ and $u$ has a leaf neighbor.  Let $uv$ be an edge in $C_g$ and $w$ be a leaf neighbor of $u$ in $SaB_u(t;b)$.

 {\bf Operation A:} The operation A  is defined  as  identifying $v$ and $u$ (denoted by $u$) and  adding a leaf vertex adjacent to $w$.  Denote by $G_A$ the graph obtained from $G$ by Operation $A$.

\begin{lemma}
\label{Le:opA}
Let $G =H(SaB_u(t;b), K_1,\dots,K_1) \in  \mathcal{U}_{n,r}$ with $1\le t\le 2, $ $3\le r \leq \frac{n+3}{2}$  and $g\geq 4$. Then
$$W(G) - W(G_A)  =  \left\{
\begin{array}{ll}
= 0  ~~~~~~\mbox{ if $ g = 7$ and $SaB_u(t,b)=SB_u(1,4)$ }; \\
< 0 ~~~~~~~~\mbox{if $g  \in \{4,5\}$};\\
> 0~~~~~~~\mbox{otherwise}.
\end{array}
\right.$$
\end{lemma}
\begin{proof}
Let $T=SaB_u(t;b)$ and $T_1=SaB_u'(t;b),$ which is obtained by adding a leaf vertex to a path of length $1$ in $SaB_u(2;b).$
By Lemma~\ref{LE:one cut vertex}, we have the following two equalities:
\begin{eqnarray*}
W(G)&=&W(C_g)+W(T)+(g-1)d(u,T)+(|T|-1)d(u,C_g),\\
W(G_A)&=&W(C_{g-1})+W(T_1)+(g-2)d(u,T_1)+|T|d(u,C_{g-1}).
\end{eqnarray*}
Note that $d(u,T_1)=d(u,T)+2$ and $d_G(w,T)=d(u,T)+|T|-2.$ Also, we can use Lemma~\ref{LE:one cut vertex} again to find $W(T_1)$ in terms of $W(T).$ We split $T_1$ such that $V(G_1)=\{v,w\},$ where $w$ is the cut vertex, and $G_2=T.$ Thus, $$W(T_1)=W(G_1)+W(T)+(|T|-1)d(w,G_1)+(2-1)d(w,T)=W(T)+d(w,T)+|T|.$$
Therefore,
\begin{eqnarray*}
W(G)-W(G_A)&=&(W(C_g)+W(C_{g-1}))+(W(T)-W(T_1))\\
&&+~((g-1)d(u,T)-(g-2)d(u,T_1))+((|T|-1)d(u,C_g)-|T|d(u,C_{g-1})\\
&=&(W(C_g)-W(C_{g-1}))-d(w,T)-|T|+d(u_1,T)+4-2g\\
&&+~((|T|-1)d(u,C_g)-|T|d(u,C_{g-1})\\
&=&(W(C_g)-W(C_{g-1}))-2|T|+6-2g+((|T|-1)d(u,C_g)-|T|d(u,C_{g-1}).
\end{eqnarray*}
Now, we have to account for whether $g$ is even or odd. Note that  by Lemma~\ref{Le:cycle} $d(u,C_g)=\frac{2}{g}W(C_g)=k^2$ when $g=2k$ and $d(u,C_g)=k(k-1)$ when $g=2k+1$. The equation becomes
$$W(G)-W(G_A) = \left\{
\begin{array}{ll}
\frac{1}{2}k^{2}-\frac{9k}{2}+(k-2)|T|+6~~~~~~\mbox{if $g = 2k$},\\
\frac{1}{2}k^{2}-\frac{9k}{2}+(k-2)|T|+4~~~~~~\mbox{if $g = 2k+1$}.
\end{array}
\right.
$$
With some simple calculation, one can show that the lemma holds.
\end{proof}

Now we can complete the proof of Theorem~\ref{TH: small r}.

{\bf Proof of Theorem~\ref{TH: small r}}. Let $G \in \mathcal{U}_{n,r}$ where $r \leq \frac{n+3}{2}$ such that $W(G)$ is minimal. By Theorem~\ref{TH: Conf} we have  that $G$ must be one of the following

(a) $G$ is a graph obtained from a triangle $xyz$  by attaching at most one leaf to each of $x$ and $y$ and a star $K_{1,t}$ to $z$ where $t$ is odd and the total number of leaves in $G$ is $n-r$.

(b) $G = H(K_2,K_1,K_1,K_1,K_1)$.

 (c)  $G =H(SaB_v(t;b), K_1,\dots,K_1)$ where $t\leq 2$, $d_G(u) = n-r + 2$ is even, and $b = d_G(u) -2$.

Note that in (a), $ 0\leq r \leq 2$ and the graph is uniquely determined. For (b), it is easy to see that $W(H(K_2,K_1,K_1,K_1,K_1)) = W(H(SB(1;2),K_1,K_1,K_1))= 24$ and $H(SB(1;2),K_1,K_1,K_1)$ is a graph described in (c).

Now we  only need to consider the case (c). By Lemma~\ref{Le:opA}, we have $g\leq 5$ or $g= 7$ and $SaB_v(t;b) = SB_v(1;b)$ is a star with $5$ vertices.

\begin{claim}
$g \leq 5$ and if $g \in \{3,4\}$,   then $SaB_v(t,b)$ must be  a star, i.e., $t = 1$.

\end{claim}
If $g = 7$ and $SaB_v(t;b)$ is a star with $5$ vertices, then $W(G_A) = W(G)$ and $W(G_A)_A < W(G_A) = W(G)$, a contradiction.

Assume $g \in \{3,4\}$. If $t = 2$, reverse Operation A on $G$. By Lemma~\ref{Le:opA}, the resulting graph has a smaller wiener index, a contradiction. This completes the proof of Theorem~\ref{TH: small r}. \hfill $\Box$


\begin{thebibliography}{s2}
\bibitem{ca2019} R.~M.~Casablanca,  and P.~Dankelmann,  Distance and eccentric sequences to bound the Wiener index, Hosoya polynomial and the average eccentricity in the strong products of graphs,  {\it Discrete Appl. Math.} 263 (2019) 105--117.

\bibitem{cavaleri2019} M.~Cavaleri, A.~Donno,  and A.~Scozzari,  Total distance, Wiener index and opportunity index in wreath products of star graphs, {\it Electron. J. Combin.} 26 (2019) no. 1, Paper 1.21, 25 pp.

\bibitem{das2016}K.~C.~Das, and  I.~Gutman, On Wiener and multiplicative Wiener indices of graphs. {\it Discrete Appl. Math.} 206 (2016) 9--14.

\bibitem{dobrynin2001} A.~Dobrynin, R.~Entringer, I.~Gutman, Wiener index of trees: Theory and applications, {\it Acta Appl. Math.} {66} (2001) 211--249.

\bibitem{u-maximum} H. Dong and B. Zhou, Maximum Wiener index of unicyclic graphs with fixed
maximum degree, {\it Ars Combinatorica} {103} (2012) 407--416.

\bibitem{furtula2013-2} B.~Furtula, I.~Gutman and H.~Lin, More trees with all degrees odd having extremal Wiener index,  {\it MATCH Commun. Math. Comput. Chem.} 70 (2013) 293--296.

\bibitem{Georgakopoulos}  A.~Georgakopoulos, S.~Wagner, Hitting times, cover cost, and the Wiener index of a tree, {\it J. Graph Theory} {84} (2017)  311--326.

\bibitem{gutman1993} I.~Gutman, A new method for the calculation of the Wiener number of acyclic molecules, {\it Journal of Molecular Structure: THEOCHEM} {285} (1993) 137--142.

 \bibitem{gutman2014} I.~Gutman, R.~Cruz, J.~Rada, Wiener index of Eulerian graphs, {\it Discrete Appl. Math.} {162} (2014) 247--250.

\bibitem{gutman2016} I.~Gutman, B.~Furtula, K.~C.~Das, On some degree-and-distance-based graph invariants of trees, {\it Appl. Math. Comput.} {289}
(2016) 1--6.

\bibitem{gutman2017} I.~Gutman, S.~Li, W.~Wei, Cacti with
n-vertices and t cycles having extremal Wiener index, {\it Discrete Appl. Math.} { 232} (2017) 189--200.

\bibitem{HLW2011} Y. Hong,  H. Liu and X. Wu,  On the Wiener index of unicyclic graphs, {\it Hacettepe Journal of Mathematics and Statistics} 40 (1) (2011) 63--68.

\bibitem{u2012} H. Hou, B. Liu, and Y. Huang. The maximum Wiener polarity index of unicyclic
graphs, {\it  Applied Mathematics and Computation}  218 (2012) 10149--10157.

\bibitem{Knor2016} M.~Knor,  R.~\v{S}krekovski, and  A.~Tepeh, Mathematical aspects of Wiener index,  {\it  Ars Math. Contemp.} 11 (2016) 327--352.

\bibitem{Knor2014} M. Knor and R. ~\v{S}krekovski. Wiener index of line graphs. Quantitative Graph
Theory: Mathematical Foundations and Applications, pages 279-301, 2014.

    \bibitem{li2006} X.-L.~Li, and I.~Gutman,  {\it Mathematical aspects of Randi\'{c}-type molecular structure descriptors. With a foreword by Milan Randi?.}  Mathematical Chemistry Monographs, 1. University of Kragujevac, Faculty of Science, Kragujevac, 2006. vi+330 pp



\bibitem{lin2014} H. Lin, Extremal Wiener index of trees with given number of vertices of even degree, {\it MATCH Commun. Math. Comput. Chem.} 72 (2014) 311--320.

 \bibitem{PB1986}O. E. Polansky and D. Bonchev. The Wiener number of graphs. i. general theory
and changes due to some graph operations,  {\it MATCH Commun. Math. Comput.
Chem.} { 21}(1986) 133--186.

\bibitem{u-diameter} S.-W. Tan,  The minimum Wiener index of unicyclic graphs
with a fixed diameter, {\it J. Appl. Math. Comput.} 56 (2018) 93--114.

\bibitem{wiener1947} H.~Wiener, Structural determination of paraffin boiling points, {\it J. Am. Chem. Soc.} { 69} (1947) 17--20.

\bibitem{Survey2014M} K. Xu, M. Liu, K. Ch. Das, I. Gutman, and B. Furtula. A survey on graphs
extremal with respect to distance-based topological indices, {\it MATCH Commun.
Math. Comput. Chem.} 71(2014) 461--508.

\bibitem{u-girth} G. Yu, L. Feng, On the Wiener index of unicyclic
graphs with given girth, {\it Ars Combinatorica}, { 94} (2010) 361--369.
\end{thebibliography}
\end{document}